\pdfoutput=1
\RequirePackage{ifpdf}
\ifpdf 
\documentclass[pdftex]{sigma}
\else
\documentclass{sigma}
\fi

\newcommand{\rmd}{\mathrm{d}}

\newcommand{\kFr}{\mathfrak{k}}

\newcommand{\Jac}{\mathrm{Jac}}
\DeclareMathOperator{\wgt}{wgt}
\DeclareMathOperator{\Complex}{\mathbb{C}}

\numberwithin{equation}{section}

\newtheorem{theo}{Theorem}[section]
\newtheorem{prop}[theo]{Proposition}

\newtheorem{Lemma}[theo]{Lemma}

{\theoremstyle{definition}
\newtheorem{rem}[theo]{Remark}

\newtheorem*{Aproblem}{Addition Problem}
\newtheorem*{Rproblem}{Reduction Problem}}

\begin{document}

\allowdisplaybreaks

\newcommand{\arXivNumber}{1912.13277}

\renewcommand{\PaperNumber}{053}

\FirstPageHeading

\ShortArticleName{Addition of Divisors on Hyperelliptic Curves via Interpolation Polynomials}

\ArticleName{Addition of Divisors on Hyperelliptic Curves \\ via Interpolation Polynomials}

\Author{Julia BERNATSKA~$^\dag$ and Yaacov KOPELIOVICH~$^\ddag$}

\AuthorNameForHeading{J.~Bernatska and Y.~Kopeliovich}

\Address{$^\dag$~National University of Kyiv-Mohyla Academy, 2 Skovorody Str., Kyiv, 04655, Ukraine}
\EmailD{\href{mailto:jbernatska@gmail.com}{jbernatska@gmail.com}}

\Address{$^\ddag$~University of Connecticut, 2100 Hillside Rd, Storrs Mansfield, 06269, USA}
\EmailD{\href{mailto:yaacov.kopeliovich@uconn.edu}{yaacov.kopeliovich@uconn.edu}}

\ArticleDates{Received February 05, 2020, in final form May 29, 2020; Published online June 14, 2020}

\Abstract{Two problems are addressed: reduction of an arbitrary degree non-special divisor to the equivalent divisor of the degree equal to genus of a curve, and addition of divisors of arbitrary degrees. The hyperelliptic case is considered as the simplest model. Explicit formulas defining reduced divisors for some particular cases are found. The reduced divisors are obtained in the form of solution of the Jacobi inversion problem which provides the way of computing Abelian functions on arbitrary non-special divisors.
An effective reduction algorithm is proposed, which has the advantage that it involves only arithmetic operations on polynomials. The proposed addition algorithm contains more details comparing with the known in cryptography, and is extended to divisors of arbitrary degrees comparing with the known in the theory of hyperelliptic functions.}

\Keywords{reduced divisor; inverse divisor; non-special divisor; generalised Jacobi inversion problem}

\Classification{32Q30; 14G50}

\section{Introduction}
In the paper we propose a method of adding arbitrary divisors on a hyperelliptic curve. We often refer to a non-special divisor, which we define by means of Riemann theorem, more accurate definition is given below in Section~\ref{ss:NSpDiv}. However, the method covers addition of special divisors as well.

Before stating the goal we introduce the notion of \emph{reduced divisor} corresponding to any non-special divisor on an algebraic curve. Suppose $g$ is genus of the curve. Any non-special divisor can be represented by a collection of points of number greater than or equal to the genus, that is $D=\sum\limits_{k=1}^{g+m} P_k$, $m \geqslant 0$. A \emph{reduced divisor} is a non-special divisor composed of $g$ points: $\widetilde{D}=\sum\limits_{k=1}^{g} P_k$.

It follows from the Riemann--Roch theorem that every non-special divisor of the form $D-(\deg D)\infty$ is equivalent to $\widetilde{D}-g \infty$, where $\widetilde{D}$ is a reduced divisor. This immediately leads to the following

\begin{Rproblem} Given a non-special divisor $D$ of degree~$g+m$, $m>0$, on an algebraic curve of genus $g$
find the corresponding reduced divisor $\widetilde{D}$ such that $D- {(g+m) \infty}$ is equivalent to $\widetilde{D} - g \infty$.
\end{Rproblem}

Addition of two special divisors can be considered as a reduction problem, if the sum forms a~non-special divisor. Non-special divisors are used in the problem statement to avoid ambiguity.

The reduction problem has a close relation to

\begin{Aproblem}Given two non-special divisors $D_1$ and $D_2$ of degrees $g+m_1$ and $g+m_2$, $m_1,m_2\geqslant 0$ respectively find a reduced divisor $\widetilde{D}$ such that $D_1+D_2 - (2g+m_1+m_2)\infty \sim \widetilde{D}-g\infty$.
\end{Aproblem}
Note that solving the reduction problem we immediately solve the addition problem since we can assume that $D_1$, $D_2$ together compose a non-special divisor $D$ of degree $2g+m_1+m_2$, and then come to the reduction problem for the new divisor $D$. On the other hand, the standard addition problem arises when the both divisors $D_1$, $D_2$ are firstly reduced to divisors $\widetilde{D}_1$, $\widetilde{D}_2$ of degree $g$ each, then addition of $\widetilde{D}_1$, $\widetilde{D}_2$ can be accomplished by any approach to addition law.

Much work has been done on the reduction problem starting with the classical work \cite{Ba}, where
the Jacobi inversion problem was solved in hyperelliptic case \cite[p.~32, Section~216]{Ba}
(briefly recalled in preliminaries). The Jacobi inversion problem is stated for a non-special divisor of degree~$g$, that is for a reduced divisor.
Points of the divisor serve as roots of two rational functions defined on a curve.
Coefficients of the rational functions are expressed through
multiply periodic functions $\wp$ evaluated at a point of Jacobian of the curve, and this point
corresponds to the divisor.
So a solution of the reduction problem coincides by the form with a solution of the Jacobi inversion problem,
which provides a method of solving the generalised Jacobi inversion problem
stated for divisor of degree greater than~$g$.

The first solution of reduction problem was given in~\cite{Ca}.
This algorithm was inspired by reduction of quadratic forms.
For low genera ($g=2,3$) many authors worked on giving more explicit solutions
to the reduction problem due to potential application in cryptography,
see~\cite{Ga} and~\cite{Su} and the literature cited there.
The explicit realisation of addition law should also have applications to the theory of heights on hyperelliptic Jacobians.
In \cite{Uc} $\wp$ functions are used to produce formulas for division polynomials on hyperelliptic curves
of low genera which was later applied in~\cite{DeJMu} to compute canonical
heights on genus $2$ curves. Though the question of division polynomials isn't treated explicitly
in the present paper it is strongly related to the reduction algorithm we propose as it is essentially
equivalent to reduced divisors of the form~$nP$ on the curve.

In preliminaries the Jacobi inversion problem is recalled and a detailed description of non-special
and special divisors is given.
The reduction problem is addressed in Sections~\ref{s:RedPr} and~\ref{s:RedAlg}.
The proposed method of reduction is based on the ideas of~\cite{bl05} and~\cite{ShK}.
A new and essential result achieved here consists in finding explicit functions
defining reduced divisors in some particular cases, which
are presented in Section~\ref{s:RedPr}.
An iterative reduction algorithm is given in Section~\ref{s:RedAlg},
as well as some comments on its application in cryptography.
Section~\ref{s:AddProc} is devoted to the addition problem.

In our setup we often suppose that points of divisors are known,
and we use them to construct polynomials and functions defining the divisors.
This approach guarantees that a divisor arising from the definition through polynomials
is located on the curve. On the other hand, the reader can forget that the polynomial coefficients were computed from points, and use symbolic notations for them. All operations are applicable to polynomials in this form as well.

In this paper we tie together two directions where addition of divisors was investigated:
hyperelliptic cryptography, and theory of Abelian functions on hyperelliptic curves.
The viewpoint of cryptographic applications describes divisors in terms of polynomials in quite an abstract manner, when the structure of divisors is left out of consideration.
Analysis of divisors by means of meromorphic functions on hyperelliptic curves helps a lot
in understanding a relation between the structure of a divisor and the form of functions which define it.
On the other hand, the addition law known in the theory of Abelian functions on hyperelliptic curves
is established for two non-special divisors of degree~$g$, which seems to be enough.
In the present paper we extend the addition law to non-special divisors of arbitrary degree.
Though we restrict our consideration to non-special divisors
and use reduction to divisors of degree~$g$, we provide some new results in
computation of Abelian functions on arbitrary non-special divisors, which arise from addition of divisors.

\section{Preliminaries}

\subsection{Hyperelliptic curve and Sato weights}
In the paper we deal with the family of hyperelliptic curves with a branch point at infinity.
A~genus $g$ curve is defined by the equation
\begin{gather}
 0=f(x,y) = -y^2 + \mathcal{P}(x) = - y^2 + \lambda_0 x^{2g+1} + \sum_{n=0}^{2g} \lambda_{4g+2-2n} x^n,\label{CHEeq1}
\end{gather}
where $\lambda_k$ are parameters of the curve, $\lambda_0=1$, and $x,y\in \Complex$.
We use Sato weights as indices, since
they are respected by the theory of rational functions, that simplifies many considerations.
Sato weight equals the opposite to the exponent of the leading term in the expansion near infinity.
Namely,
\begin{gather}\label{Param}
x=\xi^{-2}, \qquad y = \xi^{-2g-1}(1+O(\lambda)),
\end{gather}
where $\xi$ is a local parameter, and Sato weights of $x$ and $y$ are $\wgt x = 2$, $\wgt y = 2g+1$.
The weight is also assigned to every function, for example $f$ has weight $4g+2$.

\subsection{Jacobi inversion problem}
The Jacobi inversion problem gives the answer how to find $g$ points $\{(x_k,y_k)\}_{k=1}^g$
on a curve which unambiguously map into a point $u$ of Jacobian $\Jac$ of the curve.
Solution of the Jacobi inversion problem for a hyperelliptic curve is given by two rational functions
\begin{subequations}\label{JIP}
\begin{gather}
\mathcal{R}_{2g}(x;u) = x^g - \sum_{k=1}^g x^{g-k} \wp_{1,2k-1}(u),\\
\mathcal{R}_{2g+1}(x,y;u) = 2y + \sum_{k=1}^g x^{g-k} \wp_{1,1,2k-1}(u).
\end{gather}
\end{subequations}
Here multiply periodic $\wp$ functions as in \cite{Ba} are defined by
\begin{gather*}
 \wp_{i,j}(u) = - \frac{\partial^2}{\partial u_i \partial u_j} \log \sigma(u),\\
 \wp_{i,j,k}(u) = - \frac{\partial^3}{\partial u_i \partial u_j \partial u_k} \log \sigma(u)
\end{gather*}
through $g$-variable $\sigma$ function
which can be constructed by the method given in~\cite{BL2004}. More details
about rational functions can be found in \cite{bl05}.

Components of $u\in \Jac$ are indexed by Sato weights: $u=(u_1, u_3, \dots, u_{2g-1})$,
$\wgt u_n = -n$, and the standard holomorphic differentials are employed
\begin{gather*}
 \rmd u_{2k-1} = \frac{x^{g-k}\rmd x}{-2y},\qquad k=1,\dots,g.
\end{gather*}

The function $\mathcal{R}_{2g}$ is a polynomial in $x$ and has $g$ roots $\{x_k\}_{k=1}^g$.
At the same time, $\mathcal{R}_{2g}$ is rational on the curve~\eqref{CHEeq1} with $2g$ roots, namely,
$\{(x_k, y_k), (x_k,-y_k)\}_{k=1}^g$, where $\{y_k\}_{k=1}^g$ are defined by the function
$\mathcal{R}_{2g+1}$. That is divisor $D_g=\{(x_k,y_k)\}_{k=1}^{g}$ solves uniquely the following system
\begin{gather*}
 \mathcal{R}_{2g}(x;u) = 0,\qquad \mathcal{R}_{2g+1}(x,y;u) = 0,
\end{gather*}
and serves as the preimage of $u$
\begin{gather*}
\Jac \ni u=\sum_{k=1}^g \mathcal{A}(x_k,y_k)
\end{gather*}
under Abel's map
\begin{gather*}
\mathcal{A}(x,y) = \int_{\infty}^{(x,y)} \rmd u,
\end{gather*}
where $\rmd u = (\rmd u_1, \rmd u_3, \dots, \rmd u_{2g-1})^{\rm t}$.
As usual, $\mathcal{A}(D_g) = \sum\limits_{k=1}^g \mathcal{A}(x_k,y_k)$.

Note that $\mathcal{R}_{2g+1}$ has $2g+1$ roots on the curve, but only $g$ are common of two functions
$\mathcal{R}_{2g}$ and $\mathcal{R}_{2g+1}$.
Points $\{(x_k,-y_k)\}_{k=1}^{g}$ form a divisor which is inverse to $\{(x_k,y_k)\}_{k=1}^{g}$,
 and satisfy $\mathcal{R}_{2g+1}(x,-y;u) = 0$.

On the other hand, $\mathcal{R}_{2g}$ and $\mathcal{R}_{2g+1}$ can be obtained by determinant formulas
\begin{align*}
 &\mathcal{R}_{2g}(x,y;u) = x^g + \sum_{k=1}^g \alpha_{2g+2-2k} x^{k-1},&\\
 &\frac{1}{2}\mathcal{R}_{2g+1}(x,y;u) = y + \sum_{k=1}^g \beta_{2g+3-2k} x^{k-1},&
\end{align*}
similar to \eqref{RFHEeq1}, and coefficients $\alpha_n$ and $\beta_n$ are expressed in terms of
coordinates of points $\{(x_k,y_k)\}_{k=1}^{g}$.

\subsection{Non-special divisors}\label{ss:NSpDiv}
First, we recall the Riemann theorem, which is used here to define a non-special divisor.
Let $\omega$ and $\omega'$ be matrices of periods along the standard homology
$\mathfrak{a}$- and $\mathfrak{b}$-cycles, namely,
\begin{gather*}
\omega_{2k-1,j} = \oint_{\mathfrak{a}_j} \rmd u_{2k-1}, \qquad
\omega'_{2k-1,j} = \oint_{\mathfrak{b}_j} \rmd u_{2k-1},
\end{gather*}
where $k,j=1, \dots, g$, and $\rmd u_{2k-1}$ are holomorphic differentials introduced above.
In this notation $\theta\big(\omega^{-1} u; \omega^{-1}\omega'\big)$ is the Riemann theta function on Jacobian of
a curve defined by~\eqref{CHEeq1}. $K$ denotes the vector of Riemann constants. Then
\[ \theta\big(\mathcal{A}(P)-\mathcal{A}(D)-K\big)\]
as a function of a point $P$ vanishes if $D$ is a \emph{special divisor}, and has~$g$ roots if~$D$ is a \emph{non-special divisor}.

The same statement holds with multivariable $\sigma$-function, namely,
\[ \sigma\big(\mathcal{A}(P)-\mathcal{A}(D)\big)\]
as a function of $P$ vanishes if $D$ is a \emph{special divisor}, and has~$g$ roots if~$D$ is a \emph{non-special divisor}. Let $D_g=\sum\limits_{k=1}^g (x_k,y_k)$ and all points of $D_g$ are parameterised as
in~\eqref{Param} with parameters $\xi_k$. As shown in \cite[Theorem~2.7 in Russian version]{bl05}{\samepage
\begin{gather}\label{SigmaParam}
\sigma \big(\mathcal{A}(D_g)-\mathcal{A}(x(\xi),y(\xi))\big)
= \Bigg(\prod_{k=1}^g (\xi-\xi_k) \prod_{\substack{i,j=1\\ j>i}}^g (\xi_i+\xi_j) \Bigg)
\exp \big(H(-\xi,\xi_1,\dots,\xi_g;\lambda)\big),
\end{gather}
where $H(0;\lambda) = H(-\xi,\xi_1,\dots,\xi_g;0)=0$.}

From \eqref{SigmaParam} it is evident that a non-special divisor $D_g$ contains no pair of
points related by the hyperelliptic involution: $\xi_i=-\xi_j$. This is applicable to a degree $g$ divisor.
If a divisor $D_{g+m}$ of degree $g+m$ contains a pair in the hyperelliptic involution, we delete
this pair from the divisor, and do the same with all pairs in involution,
so obtain a truncated divisor $\widehat{D}_{g+m}$ equivalent to~$D_{g+m}$.
The divisor $D_{g+m}$ is \emph{non-special} if its truncated version~$\widehat{D}_{g+m}$ has a
degree equal to or greater than~$g$. In what follows divisors containing points in involution are not considered.

In a divisor $D_{g-n}$ of degree less than $g$,
say $g-n$ with $0<n<g$, the absent $n$ points are assigned to infinity, that is the corresponding parameters $\xi_k$
vanish. Putting $(x,y)$ to infinity, that is $\xi=0$, one computes $\sigma \big(\mathcal{A}(D_{g-n})\big)$,
and see that sigma function vanishes on a divisor of degree less than~$g$, and so $\wp$ functions are not defined on such a divisor. In this case non-vanishing derivatives of sigma function are used instead of $\wp$ functions. Formula~\eqref{SigmaParam} is also applicable to a divisor of degree greater than $g$, when the divisor is replaced by the equivalent reduced divisor.

\subsection{Special divisors}\label{ss:SpDiv}
In what follows special divisors are always considered as containing less than $g$ points of a curve.
\begin{prop}\label{p:HIdefDiv}
A divisor $D_m$ of degree $m$, $m<g$, is defined uniquely by the system
\begin{gather}\label{HIdefDiv}
\mathcal{H}(x)=0,\qquad y=\mathcal{I}(x)
\end{gather}
with polynomials $\mathcal{H}$ of degree $m$ and $\mathcal{I}$ of degree $m-1$ or greater, the both
vanishing on $D_m$.
\end{prop}
The proof is evident if the points are given. For example one can use determinant formulas to
construct such polynomials, at that $\mathcal{I}$ has degree $m-1$.

Now consider the system \eqref{HIdefDiv} in detail.
Equation $\mathcal{H}(x)=0$ defines $2m$ points of the curve,
namely $\{(x_k,y_k),(x_k,-y_k)\}_{k=1}^{m}$ such that $x_k$ are roots of $\mathcal{H}$.
Coordinates $y_k$ can be found from the equation of the curve~\eqref{CHEeq1},
but only $m$ points are contained in $D_m$. These points are singled out by
the second equation of the system~\eqref{HIdefDiv}. It represents a rational function
of weight $2g+1$, equal to $\wgt y$, if $\deg \mathcal{I} \leqslant g$, and so has $2g+1$ roots on the curve, $m$ of which are common with~$\mathcal{H}$. Polynomial $\mathcal{I}$ of degree greater than $g$
is also eligible.

Obviously, a divisor with points in involution can
not be covered by the definition~\eqref{HIdefDiv} since the equation for~$y$ is linear and defines only one
value of~$y$ for each~$x$.

The reduction problem is not applicable to special divisors, which are in the most reduced form.
However one can consider the addition problem for two or more special divisors, and the addition algorithm
given in Section~\ref{s:AddProc} is applicable to such divisors as well.

\section{Addition on a curve}\label{s:RedPr}

\subsection[Reduction of $g+1$ degree divisor]{Reduction of $\boldsymbol{g+1}$ degree divisor}\label{ss:Divg1}

\begin{theo}\label{T:Red1}
Let $\widetilde{D}$ be a divisor of degree $g$ such that
$\widetilde{D} - g\infty$ is equivalent to $D_{g+1}-(g+1)\infty$, where
$D_{g+1}=\sum\limits_{k=1}^{g+1} P_k$ is a non-special
divisor, and $\{P_k=(x_k,y_k)\}_{k=1}^{g+1}$ are distinct points.
Then~$\widetilde{D}$ is defined by the system
\begin{subequations}\label{JIPg1}
\begin{gather}\label{JIPg1HI}
 \mathcal{H}(x) = 0,\qquad y= - \mathcal{I}(x),
\end{gather}
where
\begin{gather}\label{JIPg1H}
 \mathcal{H}(x) = - \frac{1}{2} \sum_{\substack{k,l=1\\ l\neq k}}^{g+1}
 \frac{(y_k - y_l)^2}{(x_k-x_l)^2}
 \prod_{\substack{i=1\\ i\neq k,l}}^{g+1} \frac{(x-x_i)}{(x_l-x_i)(x_k-x_i)}
 + \sum_{n=0}^{g} x^n \sum_{j=0}^{g-n} \lambda_{2g-2n-2j} h_{j},
\end{gather}
\begin{gather}\label{JIPg1I}
 \mathcal{I}(x) = \sum_{k=1}^{g+1} y_k \frac{\prod\limits_{i=1, i\neq k}^{g+1}(x-x_i)
 - \mathcal{H}(x)}{\prod\limits_{i=1, i\neq k}^{g+1}(x_k-x_i)},
\end{gather}
\end{subequations}
and $h_n$ denotes the complete symmetric polynomial of degree $n$ in $\{x_k\}_{k=1}^{g+1}$.
\end{theo}

\begin{proof}
Define a rational function $\mathcal{R}_{2g+1}$ with $g+1$ fixed roots at
points $\{P_k=(x_k,y_k)\}_{k=1}^{g+1}$ by the determinant formula
\begin{gather}\label{RFHEeq1}
 \begin{vmatrix}
 1 & x & x^2 & \cdots & x^g & y \\
 1 & x_1 & x_1^2 & \cdots & x_1^g & y_1 \\
 1 & x_2 & x_2^2 & \cdots & x_2^g & y_2 \\
 \vdots & \vdots & \vdots & \ddots & \vdots & \vdots \\
 1 & x_{g+1} & x_{g+1}^2 & \cdots & x_{g+1}^g & y_{g+1}
 \end{vmatrix} = 0,
\end{gather}
which can be written in a different form
\begin{gather}\label{yDef1}
 y = \mathcal{G}(x)
\end{gather}
with the help of interpolation polynomial $\mathcal{G}$ of degree $g$
\begin{gather}\label{GDef1}
 \mathcal{G}(x) = \sum_{k=1}^{g+1} y_k L_k(x),
 \end{gather}
where $L_k$ have the form of Lagrange interpolating polynomials, namely
\begin{gather}
 L_k(x) = \prod_{i=1, i\neq k}^{g+1} \frac{(x-x_i)}{(x_k-x_i)}. \label{Ldef}
\end{gather}
Note that
\begin{gather*}
\sum_{k=1}^{g+1} L_k(x) = 1.
\end{gather*}

Intersection of \eqref{yDef1} with the curve
produces the unknown $g$ roots of $\mathcal{R}_{2g+1}$.
So, substitute~\eqref{GDef1} for $y$ into~\eqref{CHEeq1}, and take into account that
$L_k(x) L_j(x)$ with $k\neq j$ is divisible by $\mathcal{F}(x)=\prod\limits_{n=1}^{g+1} (x-x_n)$
\begin{gather*}
- \mathcal{G}(x)^2 + \mathcal{P}(x) =
-\Bigg(\sum_{k=1}^{g+1} y_k L_k(x)\Bigg)^2 + \mathcal{P}(x) \Bigg(\sum_{k=1}^{g+1} L_k(x)\Bigg)^2\\
\hphantom{- \mathcal{G}(x)^2 + \mathcal{P}(x)}{}
 = \sum_{k,j=1,\,k\neq j}^{g+1} \big(\mathcal{P}(x) - y_k y_j\big) L_k(x) L_j(x)
 + \sum_{k=1}^{g+1} \big(\mathcal{P}(x) - \mathcal{P}(x_k)\big) L_k(x)^2 \\
\hphantom{- \mathcal{G}(x)^2 + \mathcal{P}(x)}{}
 = \sum_{k,j=1,\,k\neq j}^{g+1} \big(\mathcal{P}(x_k) - y_k y_j\big) L_k(x) L_j(x)
 + \sum_{k=1}^{g+1} \big(\mathcal{P}(x) - \mathcal{P}(x_k)\big) L_k(x) \\
\hphantom{- \mathcal{G}(x)^2 + \mathcal{P}(x)}{}
 = \mathcal{F}(x) \Bigg(\sum_{k,j=1,\,k\neq j}^{g+1} \frac{y_k^2 - y_k y_j}
 {\prod\limits_{i=1, i\neq k}^{g+1} (x_k-x_i)\prod\limits_{i=1, i\neq j}^{g+1} (x_j-x_i)}
 \prod\limits_{\substack{i=1\\ i\neq k,j}}^{g+1} (x-x_i)\\
\hphantom{- \mathcal{G}(x)^2 + \mathcal{P}(x)=}{}
 + \sum_{k=1}^{g+1}\frac{1}{\prod\limits_{i=1, i\neq k}^{g+1} (x_k-x_i)}
 \dfrac{\mathcal{P}(x) - \mathcal{P}(x_k)}{x-x_k} \Bigg),
\end{gather*}
with an arbitrary natural $N$ it is straightforward to check
\begin{gather}\label{xProd}
 \sum_{k=1}^{N} \frac{x_k^{n}}{\prod\limits_{j=1,j\neq k}^N (x_k-x_j)} = h_{n-N+1},
\end{gather}
where $h_n$ is the complete symmetric polynomial of degree $n$ in $\{x_k\}_{k=1}^N$,
and $h_{n} = 0$ as $n<0$. Then
one finds
\begin{gather*}
 \mathcal{Q}(x) = \sum_{k=1}^{g+1} \frac{1}{\prod\limits_{\substack{i=1\\ i\neq k}}^{g+1} (x_k-x_i)}
 \dfrac{\mathcal{P}(x) - \mathcal{P}(x_k)}{x-x_k}
 = \sum_{n=0}^{g} x^n \sum_{j=0}^{g-n} \lambda_{2g-2n-2j} h_{j}.
\end{gather*}
Finally,
\begin{gather*}
 \mathcal{H}(x) = - \frac{1}{2} \sum_{\substack{k,l=1\\ l\neq k}}^{g+1}
 \frac{(y_k - y_l)^2}{(x_k-x_l)^2}
 \prod_{\substack{i=1\\ i\neq k,l}}^{g+1} \frac{(x-x_i)}{(x_l-x_i)(x_k-x_i)}
 + \mathcal{Q}(x).
\end{gather*}
Note that coefficient at $x^g$ is $\mathfrak{h}_0 = \lambda_0=1$, which arises from $\mathcal{Q}$.

Polynomial $\mathcal{H}$ has $g$ roots, say $\{\tilde{x}_k\}_{k=1}^g$, and points
$\big\{\big(\tilde{x}_k,\tilde{y}_k=\mathcal{G}(\tilde{x}_k)\big)\big\}_{k=1}^g$
give the unknown~$g$ roots of $\mathcal{R}_{2g+1}$
defined by \eqref{RFHEeq1}. Let $\mathfrak{g}_0$ be the coefficient of $\mathcal{G}$ at $x^g$, then
\begin{gather*}
 \mathcal{I}(x) = \mathcal{G}(x) - \frac{\mathfrak{g}_0}{\mathfrak{h}_0 } \mathcal{H}(x)
 = \sum_{k=1}^{g+1} y_k \frac{\prod\limits_{i=1, i\neq k}^{g+1}(x-x_i)
 - \mathcal{H}(x)}{\prod\limits_{i=1, i\neq k}^{g+1}(x_k-x_i)}.
\end{gather*}
Polynomial $\mathcal{I}$ has degree $g-1$, and $y- \mathcal{I}(x)$
vanishes at the same $g$ points as $\mathcal{H}(x)$. These points
$\big\{\big(\tilde{x}_k,\tilde{y}_k\big)\big\}_{k=1}^g$ map into $-u\in \Jac$,
which is inverse to the Abel image $u$ of the reduced divisor $\widetilde{D} = \sum\limits_{k=1}^g \widetilde{P}_k$. Therefore, the reduced divisor $\widetilde{D}$ corresponding to $D_{g+1}$
consists of points $\big\{\widetilde{P}_k=\big(\tilde{x}_k,-\tilde{y}_k\big)\big\}$.
\end{proof}

\begin{rem}\label{wpFuncg1}System \eqref{JIPg1} coincides with the solution of Jacobi inversion problem given by the rational functions \eqref{JIP}, namely
\begin{gather*}
 \mathcal{R}_{2g}(x;u) = \mathcal{H}(x),\qquad
 \mathcal{R}_{2g+1}(x,y;u) = 2y + 2\mathcal{I}(x).
\end{gather*}
Therefore, polynomials $\mathcal{H}$ and $\mathcal{I}$ allow to compute $\wp$ functions at divisor $D_{g+1}$.
\end{rem}

\subsection[Reduction of $g+2$ degree divisor]{Reduction of $\boldsymbol{g+2}$ degree divisor}\label{ss:Divg2}
\begin{theo}\label{T:Red2}
Let $\widetilde{D}$ be a divisor of degree $g$ such that
$\widetilde{D} - g\infty$ is equivalent to $D_{g+2}-(g+2)\infty$, where
$D_{g+2}=\sum\limits_{k=1}^{g+2} P_k$ is a non-special
divisor, and $\{P_k=(x_k,y_k)\}_{k=1}^{g+2}$ are distinct points. Then~$\widetilde{D}$
is defined by the system
\begin{subequations}\label{JIPg2}
\begin{gather}
 \mathcal{H}(x) = 0,\qquad y=- \mathcal{I}(x),
\end{gather}
where
\begin{gather}
 \mathcal{H}(x) = - \frac{1}{2} \sum_{\substack{k,l=1\\ l\neq k}}^{g+2}
 \frac{(y_k - y_l)^2}{(x_k-x_l)^2}
 \prod_{\substack{i=1\\ i\neq k,l}}^{g+2} \frac{(x-x_i)}{(x_l-x_i)(x_k-x_i)}
 + \sum_{n=0}^{g-1} x^n \sum_{j=0}^{g-1-n} \lambda_{2g-2-2n-2j} h_{j},
\\
\mathcal{I}(x) = \frac{1}{\mathfrak{h}_0} \Bigg(\mathfrak{g}_0 \bigg( \frac{1}{\mathfrak{h}_0} \mathcal{H}(x) -
x \sum_{n=0}^{g-1} x^n \sum_{j=0}^{g-1-n} \lambda_{2g-2-2n-2j} h_{j} \bigg) \nonumber\\
\hphantom{\mathcal{I}(x) =}{} + \frac{1}{2} \sum_{\substack{k,l=1\\ l\neq k}}^{g+2} \sum_{\substack{n=1\\ n\neq k,l}}^{g+2}
 \frac{y_n (y_k - y_l)^2}
{\prod\limits_{\substack{j=1\\ j\neq n}}^{g+2}(x_n-x_j)
 \prod\limits_{\substack{i=1\\ i\neq l}}^{g+2} (x_l-x_i) \prod_{\substack{i=1\\ i\neq k}}^{g+2}(x_k-x_i)} \nonumber \\ \hphantom{\mathcal{I}(x) =}{} \times
 \Big((x_n-x_k-x_l) \bigg(x\prod_{\substack{j=1\\ j\neq n,k,l}}^{g+2}(x-x_j)
 - \frac{1}{\mathfrak{h}_0} \mathcal{H}(x)\bigg)+
x_k x_l \prod_{\substack{j=1\\ j\neq n,k,l}}^{g+2}(x-x_j)\Big)\nonumber\\
\hphantom{\mathcal{I}(x) =}{} - \sum_{\substack{k,l=1\\ l\neq k}}^{g+2} \frac{(y_k - y_l)^2 x_l y_k }
{\prod\limits_{\substack{i=1\\ i\neq l}}^{g+2} (x_l-x_i) \prod\limits_{\substack{i=1\\ i\neq k}}^{g+2} (x_k-x_i)^2}
 \bigg(\prod_{\substack{j=1\\ j\neq k,l}}^{g+2}(x-x_j) - \frac{1}{\mathfrak{h}_0} \mathcal{H}(x)\bigg)
\Bigg),\label{JIPg2I}
\end{gather}
\end{subequations}
where $\mathfrak{h}_0$ and $\mathfrak{g}_0$ are defined by \eqref{h0} and \eqref{g0}, and $h_n$ denotes the complete symmetric polynomial of degree $n$ in $\{x_k\}_{k=1}^{g+2}$.
\end{theo}

\begin{proof}
Define a rational function $\mathcal{R}_{2g+2}$ with $g+2$ fixed roots at
points $\{P_k=(x_k,y_k)\}_{k=1}^{g+2}$ by
\begin{gather}\label{RFHEeq2}
 \begin{vmatrix} 1 & x & x^2 & \cdots & x^g & y & x^{g+1} \\
 1 & x_1 & x_1^2 & \cdots & x_1^g & y_1 & x_1^{g+1} \\
 1 & x_2 & x_2^2 & \cdots & x_2^g & y_2 & x_2^{g+1} \\
 \vdots & \vdots & \vdots & \ddots & \vdots & \vdots & \vdots \\
 1 & x_{g+2} & x_{g+2}^2 & \cdots & x_{g+2}^g & y_{g+2} & x_{g+2}^{g+1}
 \end{vmatrix} = 0.
\end{gather}
Interpolation polynomial $\mathcal{G}$ such that \eqref{RFHEeq2}
acquires the form $y=\mathcal{G}(x)$,
is defined by
\begin{gather}\label{GDef2}
 \mathcal{G}(x) = \sum_{k=1}^{g+2} y_k L_k(x),
\end{gather}
where
\begin{gather*}
 L_k(x) = \prod_{i=1, i\neq k}^{g+2} \frac{(x-x_i)}{(x_k-x_i)}.
\end{gather*}

Intersection of $y=\mathcal{G}(x)$ with the curve produces the unknown $g$ roots of $\mathcal{R}_{2g+2}$.
Substitute~\eqref{GDef2} for $y$ into \eqref{CHEeq1} and divide by $\mathcal{F}(x) = \prod\limits_{n=1}^{g+2} (x-x_n)$.
Computation similar to that given in Section~\ref{ss:Divg1} leads to
\begin{gather*}
 \mathcal{H}(x) = - \frac{1}{2} \sum_{\substack{k,l=1\\ l\neq k}}^{g+2}
 \frac{(y_k - y_l)^2}{(x_k-x_l)^2}
 \prod_{\substack{i=1\\ i\neq k,l}}^{g+2} \frac{(x-x_i)}{(x_l-x_i)(x_k-x_i)}
 + \sum_{n=0}^{g-1} x^n \sum_{j=0}^{g-1-n} \lambda_{2g-2-2n-2j} h_{j}.
\end{gather*}
In this case coefficient at $x^g$ is
\begin{subequations}
\begin{gather}\label{h0}
 \mathfrak{h}_0 = - \frac{1}{2} \sum_{\substack{k,l=1\\ l\neq k}}^{g+2}
 \frac{(y_k - y_l)^2}{(x_k-x_l)^2}
\frac{1}{ \prod\limits_{\substack{i=1\\ i\neq k,l}}^{g+2} (x_l-x_i)(x_k-x_i)} ,
\end{gather}
which does not vanish, and the coefficient at $x^{g-1}$ is
\begin{gather}
 \mathfrak{h}_1 = \lambda_0 + \frac{1}{2} \sum_{\substack{k,l=1\\ l\neq k}}^{g+2}
 \frac{(y_k - y_l)^2}{(x_k-x_l)^2}
\frac{\sum\limits_{\substack{i=1\\ i\neq k,l}}^{g+2} x_i}
{\prod\limits_{\substack{i=1\\ i\neq k,l}}^{g+2} (x_l-x_i)(x_k-x_i)}.
\end{gather}
\end{subequations}

Let $\{\tilde{x}_k\}_{k=1}^g$ be roots of polynomial $\mathcal{H}$,
then $\{\big(\tilde{x}_k,\tilde{y}_k=\mathcal{G}(\tilde{x}_k)\big)\}_{k=1}^g$
are the unknown $g$ roots of $\mathcal{R}_{2g+2}$.
In this case $\mathcal{G}$ is a polynomial of degree $g+1$, and with the help of polynomial $\mathcal{H}$
it is reduced to $\mathcal{I}$ of degree $g-1$, namely,
\begin{gather*}
 \mathcal{I}(x) = \mathcal{G}(x) - \left(\frac{\mathfrak{g}_0}{\mathfrak{h}_0}x
 - \frac{\mathfrak{g}_0 \mathfrak{h}_1}{\mathfrak{h}_0^2} + \frac{\mathfrak{g}_1}{\mathfrak{h}_0}\right) \mathcal{H}(x),
\end{gather*}
where
\begin{subequations}
\begin{gather}
 \mathfrak{g}_0 = \sum_{n=1}^{g+2} \frac{y_n}{\prod\limits_{j=1, j\neq n}^{g+2}(x_n-x_j)}, \label{g0} \\
 \mathfrak{g}_1= - \sum_{n=1}^{g+2} \frac{y_n \sum\limits_{j=1, j\neq n}^{g+2} x_j }{\prod\limits_{j=1, j\neq n}^{g+2}(x_n-x_j)}.
\end{gather}
\end{subequations}
The expression for $\mathcal{I}$ is simplified to \eqref{JIPg2I}.

Finally, the reduced divisor $\widetilde{D}$ corresponding to $D_{g+2}$
consists of points $\big\{\widetilde{P}_k=\big(\tilde{x}_k,-\tilde{y}_k\big)\big\}$.
\end{proof}

\begin{rem}\label{wpFuncg2}
Similarly to the case of $g+1$ points, system \eqref{JIPg2} coincides with the solution of Jacobi inversion problem~\eqref{JIP}, and coefficients of polynomials $\mathcal{H}$ and~$\mathcal{I}$ provide
values of $\wp$ functions at divisor $D_{g+2}$.
\end{rem}

\subsection[Reduction of $g+1$ degree divisor with duplication]{Reduction of $\boldsymbol{g+1}$ degree divisor with duplication}\label{ss:Divg1Dupl}

\begin{theo}\label{T:Red3}
Let $\widetilde{D}$ be a divisor of degree $g$ such that
$\widetilde{D} - g\infty$ is equivalent to $D_{g+1}-(g+1)\infty$, where
$D_{g+1}=2 P_1 + \sum\limits_{k=2}^{g} P_k$ is a non-special divisor, and $\{P_k=(x_k,y_k)\}_{k=1}^g$
are distinct points. Then~$\widetilde{D}$ is defined by the system
\begin{subequations}\label{JIPg1Dupl}
\begin{gather}
 \mathcal{H}(x) = 0,\qquad y= - \mathcal{I}(x),
\end{gather}
where
\begin{gather}
\mathcal{H}(x) = \sum_{n=0}^g x^n \sum_{j=0}^{g-n}
\lambda_{2g-2n-2j} h_{j}\big|_{x_{g+1}\to x_1} \nonumber\\
\hphantom{\mathcal{H}(x) =}{} - \frac{\mathcal{P}'(x_1) }{\prod\limits_{i=2}^{g} (x_1-x_i)^2}
 \Bigg( \frac{\prod\limits_{i=2}^{g} (x-x_i)-\prod\limits_{i=2}^{g} (x_1-x_i)}{(x-x_1)}
- \prod_{i=2}^{g} (x-x_i) \sum_{\iota =2}^g (x_1- x_\iota)^{-1} \Bigg) \nonumber\\
\hphantom{\mathcal{H}(x) =}{} - \frac{ \mathcal{P}'(x_1) }{2y_1 \prod\limits_{i=2}^{g} (x_1-x_i)}
 \Bigg(\sum_{k =2}^{g}
 \frac{2 y_k \prod\limits_{\substack{i=1 \\ i\neq k}}^{g} (x-x_i)}{(x_k-x_1) \prod\limits_{\substack{i=1 \\ i\neq k}}^{g} (x_k-x_i)}
 + \frac{\mathcal{P}'(x_1)}{2y_1} \frac{\prod\limits_{i=2}^{g} (x-x_i)}{\prod\limits_{i=2}^{g} (x_1-x_i)} \Bigg) \nonumber\\
\hphantom{\mathcal{H}(x) =}{} + \sum_{j=2}^{g} \frac{(y_1 - y_j)^2\prod\limits_{\substack{i=2 \\ i\neq j}}^{g} (x-x_i)}
 {(x_j-x_1) \prod\limits_{\substack{i=1 \\ i\neq j}}^{g} (x_j-x_i)\prod\limits_{i=2}^{g} (x_1-x_i)}
 \bigg(1 - (x-x_1) \sum_{\iota =2}^g (x_1 - x_\iota )^{-1} \bigg) \nonumber\\
\hphantom{\mathcal{H}(x) =}{} + \sum_{\substack{k,j=2\\ k\neq j}}^{g}
 \frac{(y_k^2 - y_k y_j) (x-x_1)^2 \prod\limits_{\substack{i=2\\ i\neq k,j}}^{g} (x-x_i) }
 {(x_k-x_1)(x_j-x_1)\prod\limits_{\substack{i=1 \\ i\neq k}}^{g} (x_k-x_i)
 \prod\limits_{\substack{i=1 \\ i\neq j}}^{g} (x_j-x_i)},\label{JIPg1HDupl}
\\
 \mathcal{I}(x) = \frac{(x-x_1)\prod\limits_{i=2}^{g} (x-x_i) - \mathcal{H}(x)}{\prod\limits_{i=2}^{g} (x_1-x_i)}
 \bigg(\frac{\mathcal{P}'(x_1)}{2y_1} - y_1 \sum_{\iota=2}^g (x_1-x_\iota)^{-1} \bigg)\nonumber\\
\hphantom{\mathcal{I}(x) =}{} + y_1 \prod_{i=2}^{g} \frac{x-x_i}{x_1-x_i}
+ \sum_{k=2}^{g} y_k \frac{(x-x_1)\prod\limits_{\substack{i=1\\ i\neq k}}^{g}(x-x_i) - \mathcal{H}(x)}
{(x_k-x_1)\prod\limits_{\substack{i=1\\ i\neq k}}^{g}(x_k-x_i)},\label{JIPg1IDupl}
\end{gather}
\end{subequations}
and $h_n$ denotes the complete symmetric polynomial of degree $n$ in $\{x_k\}_{k=1}^{g+1}$.
\end{theo}

\begin{proof}Define a rational function $\mathcal{R}_{2g+1}$ with $g+1$ fixed roots at
points $2P_1+ \sum\limits_{k=2}^g P_k$, $P_k=(x_k,y_k)$, by the determinant formula
\begin{gather}\label{RFHEeq3}
 \begin{vmatrix}
 1 & x & x^2 & \cdots & x^g & y \\
 1 & x_1 & x_1^2 & \cdots & x_1^g & y_1 \\
 1 & x_2 & x_2^2 & \cdots & x_2^g & y_2 \\
 \vdots & \vdots & \vdots & \ddots & \vdots & \vdots \\
 1 & x_{g} & x_{g}^2 & \cdots & x_{g}^g & y_{g} \\
 0 & 1 & 2 x_1 & \cdots & g x_{1}^{g-1} & y_1'
 \end{vmatrix} = 0,
\end{gather}
where $y_1'$ denotes $\rmd y_{1}/\rmd x_1$, at that $P_1=(x_1,y_1)$ is a point of the curve
defined by \eqref{CHEeq1}. Rewri\-te~\eqref{RFHEeq3} as $y=\mathcal{G}(x)$,
and find the interpolation polynomial
\begin{gather}\label{GDef3}
 \mathcal{G}(x) = \sum_{k=1}^{g} y_k \widetilde{L}_k(x)+ y_1' \widetilde{L}_{g+1}(x),
\end{gather}
where
\begin{gather*}
 \widetilde{L}_1(x) = (x-x_1)^2 \frac{\rmd}{\rmd x_1}
\Bigg( \frac{1}{(x-x_1)\prod\limits_{i=2}^{g} (x_1-x_i)} \Bigg)\prod_{i=2}^{g} (x-x_i), \\
\widetilde{L}_k(x) = \frac{(x-x_1)^2}{(x_k-x_1)^2}\prod_{i=2, i\neq k}^{g} \frac{x-x_i}{x_k-x_i},
\qquad 2 \leqslant k \leqslant g,\\
 \widetilde{L}_{g+1}(x) = (x-x_1) \prod_{i=2}^{g} \frac{x-x_i}{x_1-x_i}.
\end{gather*}
Polynomials $\widetilde{L}_k(x)$ relate to $L_k(x)$ defined by \eqref{Ldef} as follows
\begin{gather*}
\widetilde{L}_1(x) = \lim_{x_{g+1} \to x_1} \big(L_1(x) + L_{g+1}(x) \big),\\
\widetilde{L}_k(x) = \lim_{x_{g+1} \to x_1} L_k(x), \qquad 2 \leqslant k \leqslant g,\\
\widetilde{L}_{g+1}(x) = \lim_{x_{g+1} \to x_1} \frac{-L_{g+1}(x)}{\rmd \log L_{g+1}(x)/\rmd x_{g+1}}.
\end{gather*}
Note that
\begin{gather*}
\sum_{k=1}^{g} \widetilde{L}_k(x) = 1.
\end{gather*}

Solutions of $y=\mathcal{G}(x)$ which are on the curve
define the unknown $g$ roots of $\mathcal{R}_{2g+1}$.
Substitute~\eqref{GDef3} for $y$ into~\eqref{CHEeq1}, and take into account that
$\widetilde{L}_k(x) \widetilde{L}_j(x)$ with $k,j=1,\dots,g$, $k\neq j$, and $\widetilde{L}_k(x) \widetilde{L}_{g+1}(x)$
with $k=2,\dots,g$, and $\widetilde{L}_{g+1}(x)^2$ are divisible
by $\mathcal{F}(x)= (x-x_1)^2 \prod\limits_{n=2}^{g} (x-x_n)$. Actually,
\begin{gather*}
- \mathcal{G}(x)^2 + \mathcal{P}(x) =
- \bigg(\sum_{k=1}^{g} y_k \widetilde{L}_k(x) + y_1' \widetilde{L}_{g+1}(x) \bigg)^2
+ \mathcal{P}(x) \bigg(\sum_{k=1}^{g} \widetilde{L}_k(x)\bigg)^2\\
\hphantom{- \mathcal{G}(x)^2 + \mathcal{P}(x)}{}
 = \sum_{k=1}^{g} \big(\mathcal{P}(x) - \mathcal{P}(x_k)\big) \widetilde{L}_k(x)
 - (y'_1)^2 \widetilde{L}_{g+1}(x)^2 - 2 \sum_{k =1}^{g} y_k y'_1 \widetilde{L}_k(x) \widetilde{L}_{g+1}(x) \\
\hphantom{- \mathcal{G}(x)^2 + \mathcal{P}(x)=}{}
 + \sum_{\substack{k,j=1\\ k\neq j}}^{g} \big(\mathcal{P}(x_k) - y_k y_j\big) \widetilde{L}_k(x) \widetilde{L}_j(x)\\
\hphantom{- \mathcal{G}(x)^2 + \mathcal{P}(x)}{}
 = \mathcal{F}(x) \Bigg( \frac{\rmd}{\rmd x_1} \left(\frac{\mathcal{P}(x) - \mathcal{P}(x_1)}{x-x_1} \frac{1}{A(x_1)} \right)\\
\hphantom{- \mathcal{G}(x)^2 + \mathcal{P}(x)=}{}
 + \sum_{k=2}^{g} \frac{\mathcal{P}(x) - \mathcal{P}(x_k)}{(x-x_k)}
 \frac{1}{(x_k-x_1) \prod\limits_{\substack{i=1 \\ i\neq k}}^g (x_k-x_i)}\\
 \hphantom{- \mathcal{G}(x)^2 + \mathcal{P}(x)=}{}
 - \frac{\mathcal{P}'(x_1) }{A(x_1)^2} \bigg( \frac{A(x)-A(x_1)}{(x-x_1)}
 - A(x) \frac{\rmd}{\rmd x_1} \log A(x_1) \bigg) \\
 \hphantom{- \mathcal{G}(x)^2 + \mathcal{P}(x)=}{}
 - \frac{ \mathcal{P}'(x_1) }{2y_1A(x_1)} \bigg(\sum_{k =2}^{g}
 \frac{2 y_k \prod\limits_{\substack{i=1 \\ i\neq k}}^{g} (x-x_i)}{(x_k-x_1) \prod\limits_{\substack{i=1 \\ i\neq k}}^{g} (x_k-x_i)}
 + \frac{\mathcal{P}'(x_1)}{2y_1} \frac{A(x)}{A(x_1)} \bigg) \\
 \hphantom{- \mathcal{G}(x)^2 + \mathcal{P}(x)=}{}
 + \sum_{j=2}^{g} \frac{(y_1 - y_j)^2 \prod\limits_{\substack{i=2 \\ i\neq j}}^{g} (x-x_i)}
 {(x_j-x_1) \prod\limits_{\substack{i=1 \\ i\neq j}}^{g} (x_j-x_i)A(x_1)}
 \bigg(1- (x-x_1) \frac{\rmd}{\rmd x_1} \log A(x_1) \bigg) \\
 \hphantom{- \mathcal{G}(x)^2 + \mathcal{P}(x)=}{}
 + \sum_{\substack{k,j=2\\ k\neq j}}^{g}
 \frac{(y_k^2 - y_k y_j) (x-x_1)^2 \prod\limits_{\substack{i=2\\ i\neq k,j}}^{g} (x-x_i) }
 {(x_k-x_1)(x_j-x_1)\prod\limits_{\substack{i=1 \\ i\neq k}}^{g} (x_k-x_i)
 \prod\limits_{\substack{i=1 \\ i\neq j}}^{g} (x_j-x_i)}\Bigg),
\end{gather*}
where we denote $A(x) = \prod\limits_{i=2}^g (x-x_i)$, and use relation $2y_1 y_1' = \mathcal{P}'(x_1)$.
In the above computation we used the following relations
\begin{gather*}
\widetilde{L}_1(x) = \frac{A(x)}{A(x_1)} \left(1- (x-x_1) \frac{\rmd}{\rmd x_1} \log A(x_1) \right),\qquad
\widetilde{L}_{g+1}(x) = (x-x_1) \frac{A(x)}{A(x_1)}
\end{gather*}
to obtain
\begin{gather*}
\big(\mathcal{P}(x) - \mathcal{P}(x_1)\big) \widetilde{L}_1(x)
 - 2 y_1 y'_1 \widetilde{L}_1(x) \widetilde{L}_{g+1}(x)
= (x-x_1)^2 A(x) \\
\qquad{} \times
\left( \frac{\rmd}{\rmd x_1} \bigg(\frac{\mathcal{P}(x) - \mathcal{P}(x_1)}{x-x_1}
 \frac{1}{A(x_1)} \bigg)
 - \frac{\mathcal{P}'(x_1) }{A(x_1)^2} \left( \frac{A(x)-A(x_1)}{(x-x_1)}
 - A(x) \frac{\rmd}{\rmd x_1} \log A(x_1) \right) \right).
\end{gather*}

Taking the limit of \eqref{xProd} as $x_{g+1} \to x_1$ one finds
\begin{gather*}
\frac{\rmd}{\rmd x_1} \left(\frac{x_1^{n}}{A(x_1)}\right) +
\sum_{k=2}^{g} \frac{x_k^n}{(x_k-x_1) \prod\limits_{\substack{i=1\\ i\neq k}}^{n} (x_k-x_i)}
= \lim_{x_{g+1}\to x_1} h_{n-g} ,
\end{gather*}
where $h_n$ is the complete symmetric polynomial of degree $n$ in $\{x_k\}_{k=1}^{g+1}$,
and $h_{n-g}=0$ as $n<g$.
Then
\begin{gather}
 \frac{\rmd}{\rmd x_1} \Bigg(\frac{\mathcal{P}(x) - \mathcal{P}(x_1)}{x-x_1} \frac{1}{\prod\limits_{i=2}^g (x_1-x_i)} \Bigg)
 + \sum_{k=2 }^{g} \frac{\mathcal{P}(x) - \mathcal{P}(x_k)}{(x-x_k)}
 \frac{1}{(x_k-x_1)^2 \prod\limits_{i=2, i\neq k}^g (x_k-x_i)}\nonumber\\
 \qquad{} =\sum_{n=0}^g x^n \sum_{j=0}^{g-n} \lambda_{2g-2n-2j} h_j \big|_{x_{g+1}\to x_1}.\label{QpolyDupl}
\end{gather}
Taking into account that
\begin{gather*}
\frac{\rmd}{\rmd x_1} \log A(x_1) = \sum_{i=2}^g \frac{1}{x_1-x_i} ,
\end{gather*}
one obtains polynomial $\mathcal{H}$ as in \eqref{JIPg1HDupl}.
Note that coefficient at $x^g$ is $\mathfrak{h}_0=\lambda_0=1$, which arises from \eqref{QpolyDupl}.

Let $\{\tilde{x}_k\}_{k=1}^g$ be roots of polynomial $\mathcal{H}$, then points
$\big\{\big(\tilde{x}_k,\tilde{y}_k=\mathcal{G}(\tilde{x}_k)\big)\big\}_{k=1}^g$, where $\mathcal{G}$ is defined by \eqref{GDef3}, are the unknown $g$ roots of $\mathcal{R}_{2g+1}$ which is
defined by~\eqref{RFHEeq3}. Let $\mathfrak{g}_0$ be the coefficient of $\mathcal{G}$ at $x^g$,
namely,
\begin{gather*}
\mathfrak{g}_0 = \frac{1}{A(x_1)} \left( y'_1 - y_1 \sum_{\iota=2}^g (x_1-x_\iota)^{-1} \right)
+ \sum_{k=2}^{g} \frac{y_k }{(x_k-x_1)\prod\limits_{\substack{i=1\\ i\neq k}}^{g}(x_k-x_i)},
\end{gather*}
 then polynomial
\begin{gather*}
 \mathcal{I}(x) = \mathcal{G}(x) - \frac{\mathfrak{g}_0}{\mathfrak{h}_0} \mathcal{H}(x)
\end{gather*}
is of degree $g-1$. Then $y = \mathcal{I}(x)$ and $\mathcal{H}(x)=0$ define points
$\big\{\big(\tilde{x}_k,\tilde{y}_k\big)\big\}_{k=1}^g$ which map into $-u\in \Jac$,
the inverse to the Abel image $u$ of reduced divisor $\widetilde{D}$.
Therefore, the reduced divisor~$\widetilde{D}$ corresponding to $2P_1+ \sum\limits_{k=2}^g P_k$
consists of points $\big\{\widetilde{P}_k=\big(\tilde{x}_k,-\tilde{y}_k\big)\big\}$.
\end{proof}

\subsection[Reduction of divisor $(g+1)P$]{Reduction of divisor $\boldsymbol{(g+1)P}$}\label{ss:Divg1P}
\begin{theo}\label{T:Red4}
Let $\widetilde{D}$ be a divisor of degree $g$ such that
$\widetilde{D} - g\infty$ is equivalent to $D_{g+1}-(g+1)\infty$, where
$D_{g+1}=(g+1) P_1$ with non-branch point $P_1=(x_1,y_1)$. Then~$\widetilde{D}$ is defined by the system
\begin{subequations}\label{JIPg1Mult}
\begin{gather}
 \mathcal{H}(x) = 0,\qquad y= - \mathcal{I}(x),
\end{gather}
where
\begin{gather}\label{JIPg1HMult}
 \mathcal{H}(x) = (x-x_1)^g + 2 \sum_{j=0}^{g-1} (x-x_1)^j \sum_{i=0}^{j}
 \frac{y_1^{(i)} y_1^{(j+g+1-i)} }{i! (j+g+1-i)!} ,
\\
\label{JIPg1IMult}
 \mathcal{I}(x) = \sum_{n=0}^{g-1} (x-x_1)^n \left( \frac{y_1^{(n)}}{n!} -
 2 \frac{y_1^{(g)}}{g!} \sum_{i=0}^{n}
 \frac{y_1^{(i)} y_1^{(n+g+1-i)} }{i! (n+g+1-i)!} \right),
\end{gather}
\end{subequations}
and $y_1^{(n)}$ is found from
\begin{gather*}
 \mathcal{P}^{(n)}(x_1) = \sum_{k=0}^n \frac{n!}{k!(n-k)!} y_1^{(k)} y_1^{(n-k)}.
\end{gather*}
\end{theo}

\begin{proof}
Define a rational function $\mathcal{R}_{2g+1}$ with $g+1$ fixed roots at
points $(g+1) P_1$ with $P_1=(x_1, y_1)$, by the determinant formula
\begin{gather}\label{RFHEeq4}
 \begin{vmatrix}
 1 & x & x^2 & \cdots & x^g & y \\
 1 & x_1 & x_1^2 & \cdots & x_1^g & y_1 \\
 0 & 1 & 2 x_1 & \cdots & g x_1^{g-1} & y_1' \\
 \vdots & \vdots & \vdots & \ddots & \vdots & \vdots \\
 0 & 0& 0 & \cdots & g! x_1 & y_{1}^{(g-1)} \\
 0 & 0 & 0 & \cdots & g! & y_{1}^{(g)}
 \end{vmatrix} = 0,
\end{gather}
where $y_1^{(n)}$ denotes $\rmd^n y_1/\rmd x_1^n$.
Then the interpolation polynomial $\mathcal{G}$ such that \eqref{RFHEeq4}
acquires the form $y=\mathcal{G}(x)$ is defined by
\begin{gather}\label{GDef4}
 \mathcal{G}(x) = \sum_{n=0}^{g} \frac{1}{n!} (x-x_1)^n y_1^{(n)}.
\end{gather}

Points $(x,y)$ of the curve \eqref{CHEeq1}
satisfying $y=\mathcal{G}(x)$ are roots of $\mathcal{R}_{2g+1}$.
To find them substitu\-te~\eqref{GDef4} for $y$ into~\eqref{CHEeq1}, and
take into account that
\begin{gather*}
\mathcal{G}(x)^2 = \sum_{k=0}^{2g} \frac{1}{k!} (x-x_1)^k \frac{\rmd^k y_1^2}{\rmd x_1^k}
 - 2 \sum_{k=g+1}^{2g} \sum_{i=0}^{k-g-1} \frac{ (x-x_1)^k}{i! (k-i)!} y_1^{(i)} y_1^{(k-i)},
\end{gather*}
and $y_1^2 = \mathcal{P}(x_1)$.
Evidently,
\begin{gather*}
-\mathcal{G}(x)^2 + \mathcal{P}(x)
 = (x-x_1)^{2g+1} + 2 \sum_{k=g+1}^{2g} \sum_{i=0}^{k-g-1} \frac{ (x-x_1)^k}{i! (k-i)!} y_1^{(i)} y_1^{(k-i)} \\
 \hphantom{-\mathcal{G}(x)^2 + \mathcal{P}(x)}{}
 = (x-x_1)^{g+1} \left((x-x_1)^g + 2 \sum_{k=g+1}^{2g} \sum_{i=0}^{k-g-1}
 \frac{ (x-x_1)^{k-g-1}}{i! (k-i)!} y_1^{(i)} y_1^{(k-i)}\right)
 \end{gather*}
 is divisible by $\mathcal{F}(x)= (x-x_1)^{g+1}$.
 Therefore, $\mathcal{H}$ is defined by~\eqref{JIPg1HMult}, and coefficient at~$x^g$ is~$1$. Then
\begin{gather*}
 \mathcal{I}(x) = \mathcal{G}(x) - \frac{y_1^{(g)}}{g!} \mathcal{H}(x),
\end{gather*}
and one comes to \eqref{JIPg1IMult}.
\end{proof}

\subsection[Reduction of $g+m$ degree divisor]{Reduction of $\boldsymbol{g+m}$ degree divisor}\label{ss:Divgm}

Here we propose a solution of the reduction problem for a divisor of degree greater than $g+2$.
Let a divisor $D_{g+m} = \sum\limits_{k=1}^{g+m} (x_k,y_k)$ be defined by a system
\begin{gather}\label{FRFefgm}
\mathcal{F}(x) = 0,\qquad \mathcal{R}(x,y) =0,
\end{gather}
where $\mathcal{F}$ is a polynomial of degree $g+m$, and a
rational function $\mathcal{R}$ of weight $2g+m$ has the form
\begin{gather}\label{R2gm}
\mathcal{R}(x,y) = y \gamma_y(x) + \gamma_x(x),\qquad
\deg \gamma_y = [(m-1)/2],\qquad \deg \gamma_x = g+[m/2], \end{gather}
where $[\cdot]$ denotes the integer part.
If points of $D_{g+m}$ are known, $\mathcal{R}$
can be represented through a~determinant similar to \eqref{RFHEeq2}
constructed from the first $g+m+1$ elements in the list of monomials
\begin{equation*}
\big\{1, x, \dots, x^g, y, x^{g+1}, yx, \dots, x^{g+[m/2]}, y x^{[m/2]}, \dots\big\}.
\end{equation*}
Let $\kFr=[(m-1)/2]$, then
\begin{gather*}
\gamma_y(x) = \sum_{l_\kFr > \dots > l_1=1}^{g+m}
\left(\prod_{\iota=1}^{\kFr} y_{l_\iota} \right) M_{l_1,\dots,l_\kFr}(x), \\
\gamma_x(x) = \sum_{l_{\kFr+1} > \dots > l_1=1}^{g+m} (-1)^{\kFr+1}
\left( \prod_{\iota=1}^{\kFr+1} y_{l_\iota} \right) N_{l_1,\dots,l_{\kFr+1}}(x),
\end{gather*}
where $M_{l_1,\dots,l_\kFr}$ and $N_{l_1,\dots,l_{\kFr+1}}$ with repeated indices vanish, and
\begin{gather*}
M_{l_1,\dots,l_\kFr}(x) = \frac{1}{\kFr!} \prod_{\iota=1}^{\kFr} \frac{(x-x_{l_\iota})}
{\prod\limits_{\substack{i=1,\\ i\neq l_1,\dots, l_\kFr}}^{g+m} (x_{l_\iota}-x_i)}, \\
N_{l_1,\dots,l_{\kFr+1}}(x) = \frac{1}{(\kFr+1)!}\prod\limits_{\substack{i=1,\\ i\neq l_1,\dots, l_{\kFr+1}}}^{g+m}
\frac{(x-x_i)} {\prod\limits_{\iota=1}^{\kFr+1} (x_{l_{\iota}}-x_i)} .
\end{gather*}
Note that
\begin{gather*}
\sum_{\substack{l_{\iota}=1\\ l_{\iota} \neq l_1,\dots,\widehat{l}_\iota,\dots,l_{\kFr}}}^{g+m}
M_{l_1,\dots,\widehat{l}_\iota,\dots,l_{\kFr}}(x) = 0, \\
\sum_{\substack{l_{\iota}=1\\ l_{\iota} \neq l_1,\dots,\widehat{l}_\iota,\dots,l_{\kFr+1}}}^{g+m}
N_{l_1,\dots,l_\iota,\dots,l_{\kFr+1}}(x) = \frac{1}{\kFr+1}M_{l_1,\dots,\widehat{l}_\iota,\dots,l_{\kFr+1}}(x),
\end{gather*}
where $\widehat{l}_\iota$ denotes the eliminated index. Evidently, summation of
$N_{l_1,\dots,l_\iota,\dots,l_{\kFr+1}}$ over two or more indices brings to zero.

Instead of interpolation polynomial the following rational function~$\mathcal{G}$ is used
\begin{gather}\label{GDefm}
 \mathcal{G}(x) = - \frac{\gamma_x(x)}{\gamma_y(x)}.
\end{gather}
Substitution of \eqref{GDefm} for $y$ into $f$ from \eqref{CHEeq1}
leads to
\begin{equation*}
-\gamma_x(x)^2 + \gamma_y(x)^2 \mathcal{P}(x),
\end{equation*}
which is divisible by $\mathcal{F}$
due to the construction, and the quotient polynomial $\mathcal{H}$
has degree $g$, namely,
\begin{gather}\label{Hdivgm}
\mathcal{H}(x) = \frac{\gamma_y(x)^2 \mathcal{P}(x) -\gamma_x(x)^2 }{\mathcal{F}(x)}.
\end{gather}
Finally, the reduced divisor corresponding to $D_{g+m}$ is defined by
\begin{gather}\label{DivgmHG}
\mathcal{H}(x)=0,\qquad y = - \mathcal{G}(x),
\end{gather}
where $\mathcal{G}$ is given by \eqref{GDefm}.
If the form $y+\mathcal{I}(x)$ as in \eqref{JIPg1HI}
with polynomial $\mathcal{I}$ of degree $g-1$ is required, it can be constructed by the formula
\begin{gather}\label{SLAE}
y+\mathcal{I}(x) = \mathcal{H}(x) \big(y\nu_y(x) + \nu_x(x)\big)
+ \big(y\gamma_y(x) - \gamma_x(x)\big) \mathcal{M}(x),\\
\deg \nu_y = \deg \gamma_y-1,\qquad
\deg \nu_x = \deg \gamma_x-1,\qquad
\deg \mathcal{M} = g-1. \notag
\end{gather}
Unknown coefficients of polynomials $\nu_y$, $\nu_x$, and $\mathcal{M}$ of number $2g+\big[m-\frac{1}{2}\big]$
are found from the same number of equations arising as vanishing coefficients at monomials
$\big\{y x^{g+[(m-1)/2] -1},\ldots,\allowbreak y x, x^{2g+[m/2]-1}, \ldots, x^g\big\}$ and the unit coefficient at $y$.

\begin{rem}
In the definition \eqref{FRFefgm} of divisor $D_{g+m}$ the rational function $\mathcal{R}$ has weight $2g+m$,
and this is the minimal weight of a function whose $g+m$ roots can be chosen arbitrarily.
This function is required in order to obtain a polynomial of degree $g$ in \eqref{Hdivgm}.

Below we consider also the definition of a divisor by two polynomials \eqref{FLsyst}, which usually occurs in
cryptography oriented papers. In this case function $y-\mathcal{L}(x)$ has weight $2g+2m-2$, evidently it is
greater than the minimal if $m>2$. This means that the intersection of this function with the curve contains $2g+2m-2$ points, which is seen by substituting $\mathcal{L}$ for $y$ into $f(x,y)=0$.
So the complement divisor to $D_{g+m}$ in the intersection has degree $g+m-2$, and with $m>2$ formulas \eqref{GDefm}--\eqref{DivgmHG} with $\mathcal{G}$ replaced by $\mathcal{L}$ do not lead to a reduced divisor.
\end{rem}

\section{The reduction algorithm} \label{s:RedAlg}
Reduction of a degree $g+m$ divisor is realised directly in Section~\ref{ss:Divgm}, which is applicable to
a~divisor defined by \eqref{FRFefgm} or given as a collection of points. This realisation involves solution
of a system of linear algebraic equations \eqref{SLAE}. To avoid this type of computation
we suggest an iterative algorithm of reduction, involving only arithmetic operations on polynomials.

Recall that we deal with the
hyperelliptic curve defined by \eqref{CHEeq1}.
Let $D_{g+m} = \sum\limits_{k=1}^{g+m} P_k$ be a divisor to reduce, where $m \geq 1$,
and $P_k=(x_k,y_k)$. Reduction consists in finding a~divisor $\widetilde{D}_g =\sum\limits_{k=1}^g\widetilde{P}_k$ such that $\widetilde{D}_g - g\infty$ is equivalent to $D_{g+m}-(g+m)\infty$. The reduced divisor is defined by
\begin{itemize}\itemsep=0pt
\item a polynomial $\mathcal{H}(x)$ of degree $g$, vanishing at $\widetilde{P}_k=(\tilde x_k,\tilde y_k)$,
\item and an interpolation polynomial $\mathcal{I}(x)$ of degree $g-1$,
such that $\tilde{y}_k = \mathcal{I}(\tilde{x}_k)$.
\end{itemize}
The pair of polynomials $\mathcal{H}$, $\mathcal{I}$ defines divisor $\widetilde{D}_g$ uniquely
by the system
\begin{gather*}
 \mathcal{H}(x)=0, \qquad y = \mathcal{I}(x).
\end{gather*}

In \cite{bl05} and \cite{ShK} the close addition problem is
solved by means of the determinant construction.
Here we suggest a more effective solution.

Let a divisor $D$ of degree $g+m$ with $m>0$ be given.
\begin{enumerate}\itemsep=0pt
\item[(I)] If $\deg(D)=g+1$ the result is given by polynomials $\mathcal{H}$ and $\mathcal{I}$
defined by \eqref{JIPg1} if all $g+1$ points are distinct, by \eqref{JIPg1Dupl} if $P_{g+1}=P_1$,
and by \eqref{JIPg1Mult} in the case of $D=(g+1)P_1$.
\item[(II)] If $\deg(D)=g+2$ the result is given by polynomials $\mathcal{H}$ and $\mathcal{I}$
defined by \eqref{JIPg2}.
\item[(III)] If $\deg (D)>g+2$ then an iterative procedure is used.
\end{enumerate}
The iterative procedure is the following.
Dealing with a divisor $D$ of degree $g+m$ with $m>2$, one performs the following steps:

\textbf{Step 1.}
Start with any $g+1$ points, say $\{P_k=(x_k,y_k)\}_{k=1}^{g+1}$,
 and find polynomials $\mathcal{H}^{(1)}$, $\mathcal{I}^{(1)}$ by
 formulas \eqref{JIPg1H} and \eqref{JIPg1I}, or \eqref{JIPg1HDupl} and \eqref{JIPg1IDupl},
 or \eqref{JIPg1HMult} and \eqref{JIPg1IMult}. Then relations
\begin{gather*}
 \mathcal{H}^{(1)}(x) =0,\qquad
 y =- \mathcal{I}^{(1)}(x)
\end{gather*}
define $g$ points $\big\{\widetilde{P}_k^{(1)}=\big(\tilde{x}_k^{(1)},\tilde{y}_k^{(1)}\big)\big\}_{k=1}^g$
on curve \eqref{CHEeq1},
which replace the chosen $g+1$ points of the divisor $D$. In this way a new divisor $D_{g+m-1}$
of degree $g+m-1$ is constructed.

\textbf{Step 2.} Suppose that a divisor $D_{g+m-l}$ of degree $g+m-l$ is found, which consists of $g$ points
 $\big\{\widetilde{P}_k^{(l)}=\big(\tilde{x}_k^{(l)},\tilde{y}_k^{(l)}\big)\big\}_{k=1}^g$
 defined by
 \begin{gather}\label{HIgDef}
 \mathcal{H}^{(l)}(x) =0,\qquad
 y =- \mathcal{I}^{(l)}(x),
\end{gather}
 and the remaining $m-l$ points of $D$ which form a divisor $D'_{m-l}$. Let $P_{g+l+1}$ be a point from~$D'_{m-l}$. With a collection of points $\big\{\widetilde{P}_k^{(l)}\big\}_{k=1}^g \cup \{P_{g+l+1}\}$ construct new polynomials
\begin{subequations}\label{FGlDef}
\begin{gather}
\mathcal{F}^{(l)}(x) = (x-x_{g+l+1}) \mathcal{H}^{(l)}(x), \label{FdefNext} \\
\mathcal{G}^{(l)}(x) = \mathcal{I}^{(l)}(x)
 + \big(y_{g+l+1} - \mathcal{I}^{(l)}(x_{g+l+1})\big)
 \dfrac{\mathcal{H}^{(l)}(x)}{\mathcal{H}^{(l)}(x_{g+l+1})} \label{GdefNext}
\end{gather}
\end{subequations}
of degrees $g+1$ and $g$, respectively. Evidently, the system{\samepage
\begin{gather}\label{FGeqs}
 \mathcal{F}^{(l)}(x) = 0,\qquad y = \mathcal{G}^{(l)}(x)
\end{gather}
has $g+1$ solutions at points $\big\{\widetilde{P}_k^{(l)}\big\}_{k=1}^g \cup \{P_{g+l+1}\}$.}

\textbf{Step 3.} Next, reduce the polynomials $\mathcal{F}^{(l)}$, $\mathcal{G}^{(l)}$ to
polynomials $\mathcal{H}^{(l+1)}$ and $\mathcal{I}^{(l+1)}$ of deg\-rees~$g$ and $g-1$
\begin{subequations}
\begin{gather}
\mathcal{H}^{(l+1)}(x) = \frac{\mathcal{P}(x) - \mathcal{G}^{(l)}(x)^2}{\mathcal{F}^{(l)}(x)}, \label{HdefNext} \\
\mathcal{I}^{(l+1)}(x) = \mathfrak{g}_0 \mathcal{H}^{(l+1)}(x)-\mathcal{G}^{(l)}(x), \label{IdefNext}
\end{gather}
\end{subequations}
where $\mathfrak{g}_0$ is the coefficient of $x^g$ in $\mathcal{G}^{(l)}(x)$.
Note that $\mathcal{P}(x) - \mathcal{G}^{(l)}(x)^2$ is divisible by $\mathcal{F}^{(l)}(x)$
due to \eqref{FGeqs} and the fact that $\big\{\widetilde{P}_k^{(l)}\big\}_{k=1}^g \cup \{P_{g+l+1}\}$
are points of the curve $y^2=\mathcal{P}(x)$. The system
\begin{gather*}
 \mathcal{H}^{(l+1)}(x) = 0,\qquad y = \mathcal{I}^{(l+1)}(x)
\end{gather*}
defines $g$ points $\big\{\widetilde{P}_k^{(l+1)}=\big(\tilde{x}_k^{(l+1)},\tilde{y}_k^{(l+1)}\big)\big\}_{k=1}^g$,
which together with the remaining $m-l-1$ points of $D$ form a new divisor $D_{g+m-l-1}$.
If $l+1< m$, return to Step~2.
This step is the same as the reduction algorithm given by Cantor~\cite{Ca}.

The iterative procedure described above uses reduction by $g+1$ points at each step.
One can use a different strategy, for example with $m=2\kappa$ the shortest iterative process is to apply
reduction by $g+2$ points $\kappa$ times, and with $m=2\kappa+1$ is to apply one reduction by $g+1$ points and $\kappa$ reductions by $g+2$ points.

\begin{rem}Note that computation of $\mathcal{I}^{(l)}(x)$ at each step is unnecessary.
It is enough to replace \eqref{GdefNext} by
\begin{gather*}
\mathcal{G}^{(l)}(x) = -\mathcal{G}^{(l-1)}(x)
 + \big(y_{g+l+1} + \mathcal{G}^{(l-1)}(x)\big)
 \dfrac{\mathcal{H}^{(l)}(x)}{\mathcal{H}^{(l)}(x_{g+l+1})},
\end{gather*}
then \eqref{IdefNext} can be skipped.
\end{rem}

\begin{rem}
The reduction algorithm is expressed in terms of divisors since this explanation is geometric and the most evident.
Despite the description, finding divisor is not required. Instead, coefficients of polynomials
$\mathcal{H}^{(l)}$ and $\mathcal{I}^{(l)}$ completely define the intermediate divisor of degree $g$,
as well as polynomials $\mathcal{F}^{(l)}$ and $\mathcal{G}^{(l)}$ define the intermediate divisor of degree $g+1$.
So the algorithm can be realised over any field.
\end{rem}

\begin{rem}
Let us point out that Step 2 of the reduction algorithm provides a procedure of adding one point
to a divisor of degree $g$, or to any greater degree divisor defined by a system in the form \eqref{HIgDef}.
So the reduction algorithm solves the problem of adding a non-special divisor
defined by a pair of polynomials and a special divisor given as a collection of points.
A collection of points is added by one point according to the algorithm.

Another approach to addition is presented in the next section.
\end{rem}

\begin{rem}
One can add one point to a special divisor using formulas \eqref{FGlDef}.
The special divisor is supposed to be defined as in Proposition~\ref{p:HIdefDiv}.
Then~\eqref{FGeqs} represents the resulting divisor directly.
\end{rem}

\textbf{Application in cryptography.} 
We would like to point out two practical setups
which serve as hyperelliptic cryptography algorithms.

I.\ Alice and Bob choose a non-special divisor $D_g$, which is public.
Alice chooses a~number~$n_A$ and keeps in secret, and Bob chooses a number $n_B$
and keeps in secret. Next, they are using the reduction algorithm described above to obtain a reduced form for the divisors~$n_A D_g$ and~$n_B D_g$. These are the divisors they exchange.
Once Alice gets the divisor~$n_B D_g$ from Bob (reduced form), she computes $n_A(n_B D_g)$.
On the other hand, Bob computes the same divisor as~$n_B(n_A D_g)$.
The reduced form of this divisor is the shared secret of Bob and Alice.
This is an implementation of the Diffie--Hellman exchange with the help of the reduction algorithm introduced above.

The first step of the reduction algorithm for a scalar multiple $n_j D_g$
is provided by Theo\-rem~\ref{T:Red3}.
Divisor $D_g$ is assumed to consist of distinct points, and
only on the first step a~collection of $g+1$ points contains two equal points. It is very unlikely, that
a reduced divisor on any step contains points coinciding with $D_g$.
Points of the divisor $D_g$ can be chosen
from a~desired finite field to produce polynomials with coefficients from this field.

II.\ Alice and Bob choose a non-special divisor $D_g$ of the form $g P_1$.
Then Alice computes the reduced form of $n_A (g P_1)$,
and Bob computes the reduced form of $n_B (g P_1)$.
They exchange the reduced forms of their divisors.
The first step of the reduction algorithm in this case is provided by Theorem~\ref{T:Red4},
on further steps Theorem~\ref{T:Red1} or~\ref{T:Red2} is used.
After exchange Alice and Bob compute the reduced form of divisor $n_A n_B (g P_1)$
which is the shared secret.

\section{The addition algorithm} \label{s:AddProc}
Now we suggest a procedure to solve the addition problem.
We start with a different setup. Let two non-special divisors~$D_{g+m_1}$ and $D_{g+m_2}$
of degrees $g+m_1$ and $g+m_2$
are defined by two polynomials each. Namely,
polynomials $\mathcal{F}_1$ and $\mathcal{L}_1$ of degrees
$g+m_1$ and $g+m_1-1$ define~$D_{g+m_1}$ by the system
\begin{gather*}
\mathcal{F}_1(x) = 0,\qquad y = \mathcal{L}_1(x);
\end{gather*}
and polynomials $\mathcal{F}_2$ and $\mathcal{L}_2$ of degrees
$g+m_2$ and $g+m_2-1$ define $D_{g+m_2}$ by the system
\begin{gather*}
\mathcal{F}_2(x) = 0,\qquad y = \mathcal{L}_2(x).
\end{gather*}
Degrees of polynomials $\mathcal{L}_1$ and $\mathcal{L}_2$ are justified as follows.
Suppose we are
given points $\sum\limits_{k=1}^{g+m}\! (x_k,y_k)\allowbreak = D_{g+m}$, then
a polynomial $\mathcal{L}$ such that $y-\mathcal{L}(x)$ vanishes on $D_{g+m}$
is constructed by the determinant formula on the monomials $\big\{1,x,\dots,x^{g+m-1},y\big\}$. Since $(x_k,y_k)$ are points of a~hyperelliptic curve there is no linear relation between coordinates~$y_k$, and so the coefficient at~$x^{g+m-1}$ does not vanish. Obviously,
$y-\mathcal{L}(x)$ is not the minimal weight function to define~$D_{g+m}$.

\begin{Lemma}\label{L:MinF}
Let a non-special divisor of degree $D_{g+m}$ is defined by two
polynomials: $\mathcal{F}$ of degree $g+m$ and $\mathcal{L}$ of degree equal to or greater than $g+m-1$
as follows
\begin{gather}\label{FLsyst}
\mathcal{F}(x) = 0,\qquad y = \mathcal{L}(x).
\end{gather}
Then a rational function $\mathcal{R}$ of the minimal weight $2g+m$ exists,
and the system
\begin{gather*}
\mathcal{F}(x) = 0,\qquad \mathcal{R}(x,y)=0
\end{gather*}
defines $D_{g+m}$ equivalently.
\end{Lemma}
\begin{proof}
The rational function $\mathcal{R}$ of weight $2g+m$ has the form \eqref{R2gm}.
If $\deg \mathcal{L} \geqslant \deg \mathcal{F}$, then~$\mathcal{L}$ in~\eqref{FLsyst}
can be replaced by $\widetilde{\mathcal{L}} = \mathcal{L} \!\!\mod \mathcal{F}$.
Suppose $\deg \mathcal{L} = g+m-1$, then $\mathcal{R}$ is constructed in the form
\begin{gather*}
\begin{split}&
\mathcal{R}(x,y) = \mathcal{N}(x) \mathcal{F}(x) +
\big(y-\mathcal{L}(x)\big) \mathcal{M}(x),\\
& \deg \mathcal{N} = [(m-1)/2]-1, \qquad \deg \mathcal{M} = [(m-1)/2].\end{split}
\end{gather*}
On the right hand side extra powers of $x$ arises, namely from $x^{g+[m/2]+1}$ to $x^{g+m-1+[(m-1)/2]}$,
coefficients of which should vanish. Thus, for $2[(m-1)/2]+1$ unknowns one obtains equations of number
\[ m+[(m-1)/2]-[m/2] -1 = 2[(m-1)/2].\]
So the function $\mathcal{R}$ up to a constant multiple is found.
\end{proof}

\begin{Lemma}\label{L:AddL}
Let $\mathcal{F}_1$, $\mathcal{L}_1$ and $\mathcal{F}_2$, $\mathcal{L}_2$
be two pairs of polynomials, at that $\deg \mathcal{L}_i < \deg \mathcal{F}_i$
and $\gcd(\mathcal{F}_1,\mathcal{F}_2)=1$.
Then a polynomial $\mathcal{L}$ exists such that $\mathcal{L}_i = \mathcal{L} \!\!\mod \mathcal{F}_i$, $i=1,2$.
\end{Lemma}
The lemma is known as the Chinese remainder theorem, see for example \cite[Theorem 16.19]{Shoup}.
For the sake of completeness a proof is provided.
\begin{proof}
As $\gcd(\mathcal{F}_1,\mathcal{F}_2)=1$, there exist polynomials
$\mathcal{M}_1$ and $\mathcal{M}_2$ such that
\begin{equation}\label{MFrel}
\mathcal{M}_2(x)\mathcal{F}_1(x)+\mathcal{M}_1(x)\mathcal{F}_2(x)=1.
\end{equation}
A polynomial $\mathcal{M}_i$ can be taken of degree $\mathcal{F}_i-1$,
then $\deg \mathcal{F}_1+\deg \mathcal{F}_2$ unknown coefficients are found by
equating coefficients on the both sides of the relation \eqref{MFrel}.
Then polynomial
\begin{equation}\label{LdefAdd}
\mathcal{L}(x) = \mathcal{M}_2(x)\mathcal{F}_1(x) \mathcal{L}_2(x)
+ \mathcal{M}_1(x)\mathcal{F}_2(x) \mathcal{L}_1(x)
\end{equation}
satisfies the lemma conditions. Indeed,
\begin{align*}
\mathcal{L}(x) - \mathcal{L}_1(x)
& =\mathcal{M}_2(x)\mathcal{F}_1(x) \mathcal{L}_2(x)
+ \mathcal{M}_1(x)\mathcal{F}_2(x) \mathcal{L}_1(x) - \mathcal{L}_1(x) \\
& = \mathcal{M}_2(x)\mathcal{F}_1(x) \big(\mathcal{L}_2(x) - \mathcal{L}_1(x) \big)
\end{align*}
is divisible by $\mathcal{F}_1$ as required. The same is true for $\mathcal{F}_2$.
\end{proof}

The algorithm of additing two divisors $D_{g+m_1}$ and $D_{g+m_2}$ consists of the following steps
\begin{enumerate}\itemsep=0pt
\item The sum $D_{2g+m_1+m_2}$ of $D_{g+m_1}$ and $D_{g+m_2}$ is uniquely defined by
two polynomials~$\mathcal{F}$ and~$\mathcal{L}$ such that $\mathcal{F}(x) = \mathcal{F}_1(x) \mathcal{F}_2(x)$
at that $\gcd(\mathcal{F}_1,\mathcal{F}_2)\,{=}\,1$,
and $\mathcal{L}$ is obtained by \eqref{LdefAdd}. Note that
$\deg \mathcal{L} = 3g + m_1+m_2+\max(m_1,m_2)-2$. As seen from the proof of Lemma~\ref{L:AddL},
$\mathcal{L}$~vanishes on the both divisors.
This step coincides with a special case of Cantor's composition algorithm \cite[p.~98]{Ca}.
\item By Lemma~\ref{L:MinF} a rational function $\mathcal{R}$ of weight $2g+m_1+m_2$
vanishing on $D_{2g+m_1+m_2}$ is constructed from $\mathcal{F}$ and $\mathcal{L}$.
\item The reduced problem for $D_{2g+m_1+m_2}$ defined by the polynomial $\mathcal{F}$
and the rational function~$\mathcal{R}$ is solved as in Section~\ref{ss:Divgm}.
\end{enumerate}
Another approach to solution of the addition problem for two divisors of degree $g$ can be found in
\cite[Theorem 1.23, p.~75--76 in Russian version]{bl05}.

\begin{rem}The addition algorithm is also applicable to special divisors of degree $m<g$
provided that their sum is a non-special divisor of degree greater than~$g$.
If the resulting divisor has degree equal to~$g$ or less
(and has no points in involution as we explained in preliminaries),
only Step~1 is needed.
\end{rem}

\section{Conclusion}
The reduction problem is solved explicitly for divisors of degrees $g + 1$ and $g + 2$
in the case of all distinct points (Theorems~\ref{T:Red1} and \ref{T:Red2}),
for a divisor of degree $g + 1$ with duplication (Theorem~\ref{T:Red3}), for
a divisor of the form $(g + 1)P$ (Theorem~\ref{T:Red4}).
Polynomials defining a reduced divisor are expressed in terms of the points of an initial divisor,
and their coefficients are computed directly.

It is worth to note that the mentioned polynomials also serve as a solution of the Jacobi inversion problem
for a reduced divisor, see Remarks~\ref{wpFuncg1} and \ref{wpFuncg2}, similar relations hold in the other cases.
And so the polynomial coefficients give values $\wp\big(u(D)\big)$ on an initial divisor~$D$ for~$2g$ functions which form a basis of the differential field of Abelian functions on Jacobian of the curve.
The demand for such values arises in some problems of mathematical physics.
Until now this approach to computation of $\wp$ functions has not appeared in the literature.

The reduction problem introduces the relation of equivalence on the space of non-special divisors on a curve.
To every non-special divisor an equivalent reduced divisor is assigned, the latter serves as a representative of an equivalence class consisting of all divisors reduced to this representative reduced divisor.
A reduced divisor maps uniquely to a point of Jacobian of the curve, and its equivalence class
maps to the same point of Jacobian. So a many-to-one mapping from the space of non-special divisors
to Jacobian arises. This idea can be used to compute $\wp$ functions on arbitrary non-special divisors
and solve the generalised Jacobi inversion problem.

The proposed iterative reduction algorithm has the advantage that all steps are realised in terms of polynomials obtained by means of arithmetic operations of addition, multiplication and division, and so the algorithm preserves the field to which coefficients of initial polynomials belong. The initial divisor is supposed to belong to the same field. The algorithm can also be interpreted as addition of a non-special divisor defined by a pair of polynomials and a special or non-special divisor given as a collection of points. Two scenarios of hyperelliptic cryptography algorithms on its base were suggested.

A solution of the reduction problem which does not involve points is also given
for a~deg\-ree~$g+m$ divisor. Two ways to define the divisor are considered: by two polynomials,
and by a~polynomial and a rational function of the minimal weight.
The relation between these two types of definition is described, as well as the necessity to use
the rational function of the minimal weight in order to find the reduced divisor.
And the proposed addition algorithm, whose first step is the standard addition algorithm producing two polynomials for the resulting divisor, is completed with finding a rational function of the minimal weight
and the reduction problem.

\subsection*{Acknowledgements}
The authors are thankful to the referees for the comments which had improved the paper substantially.

\pdfbookmark[1]{References}{ref}
\LastPageEnding


\begin{thebibliography}{99}
\footnotesize\itemsep=0pt

\bibitem{Ba}
Baker H.F., Abelian functions. Abel's theorem and the allied theory of theta
 functions, \textit{Cambridge Mathematical Library}, Cambridge University Press,
 Cambridge, 1995.

\bibitem{BL2004}
Bukhshtaber V.M., Leykin D.V., Heat equations in a nonholomic frame,
 \href{https://doi.org/10.1023/B:FAIA.0000034039.92913.8a}{\textit{Funct. Anal. Appl.}} \textbf{38} (2004), 88--101.

\bibitem{bl05}
Bukhshtaber V.M., Leykin D.V., Addition laws on {J}acobians of plane algebraic
 curves, \textit{Proc. Steklov Inst. Math} \textbf{251} (2005), 49--120.

\bibitem{Ca}
Cantor D.G., Computing in the {J}acobian of a hyperelliptic curve,
 \href{https://doi.org/10.2307/2007876}{\textit{Math. Comp.}} \textbf{48} (1987), 95--101.

\bibitem{DeJMu}
de~Jong R., M\"{u}ller J.S., Canonical heights and division polynomials,
 \href{https://doi.org/10.1017/S0305004114000371}{\textit{Math. Proc. Cambridge Philos. Soc.}} \textbf{157} (2014), 357--373,
 \href{https://arxiv.org/abs/1306.4030}{arXiv:1306.4030}.

\bibitem{Ga}
Gaudry P., Fast genus 2 arithmetic based on theta functions, \href{https://doi.org/10.1515/JMC.2007.012}{\textit{J.~Math.
 Cryptol.}} \textbf{1} (2007), 243--265.

\bibitem{ShK}
Shaska T., Kopeliovich Y., Additiona laws on {J}acobians from a geometric point
 of view, \href{https://arxiv.org/abs/1907.11070}{arXiv:1907.11070}.

\bibitem{Shoup}
Shoup V., A computational introduction to number theory and algebra, 2nd~ed.,
 Cambridge University Press, Cambridge, 2009.

\bibitem{Su}
Sutherland A.V., Fast {J}acobian arithmetic for hyperelliptic curves of genus~3, in Proceedings of the {T}hirteenth {A}lgorithmic {N}umber {T}heory
 {S}ymposium, \textit{Open Book Ser.}, Vol.~2, \href{https://doi.org/10.2140/obs.2019.2.425}{Math. Sci. Publ.}, Berkeley, CA, 2019, 425--442, \href{https://arxiv.org/abs/1607.08602}{arXiv:1607.08602}.

\bibitem{Uc}
Uchida Y., Division polynomials and canonical local heights on hyperelliptic
 {J}acobians, \href{https://doi.org/10.1007/s00229-010-0394-9}{\textit{Manuscripta Math.}} \textbf{134} (2011), 273--308.

\end{thebibliography}
\end{document}